\documentclass[final,leqno,onetabnum]{siamltex}
\usepackage{stmaryrd}

\usepackage{amsmath}
\usepackage{amsfonts}
\usepackage{amssymb}
\usepackage{epsfig}
\usepackage{latexsym}
\usepackage{exscale}
\usepackage{subfigure}

\usepackage{graphicx}

\usepackage{color}
\usepackage[normalem]{ulem}

\def\k{\kern .5em}
\def\er{\kern .2em}

\newcommand{\be}{\begin{equation}}
\newcommand{\ee}{\end{equation}}
\newcommand{\ba}{\begin{array}}
\newcommand{\ea}{\end{array}}
\newcommand{\bea}{\begin{eqnarray*}}
\newcommand{\eea}{\end{eqnarray*}}
\newcommand{\bean}{\begin{eqnarray}}
\newcommand{\eean}{\end{eqnarray}}

\newtheorem{Algorithm}{Algorithm}[section]

\newcommand{\lc}{\mathrel{\raise2pt\hbox{${\mathop<\limits_{\raise1pt\hbox{\mbox{$\sim$}}}}$}}}
\newcommand{\gc}{\mathrel{\raise2pt\hbox{${\mathop>\limits_{\raise1pt\hbox{\mbox{$\sim$}}}}$}}}
\newcommand{\ec}{\mathrel{\raise1pt\hbox{${\mathop=\limits_{\raise2pt\hbox{\mbox{$\sim$}}}}$}}}

\newcommand{\Label}[1]{\label{#1}{{\mbox{$\;$\fbox{\tiny\tt #1}$\;$}}}}
\renewcommand{\Label}[1]{\label{#1}}
\renewcommand{\theequation}{\arabic{section}.\arabic{equation}}
\renewcommand{\thetheorem}{\arabic{section}.\arabic{theorem}}
\renewcommand{\thelemma}{\arabic{section}.\arabic{lemma}}
\renewcommand{\theproposition}{\arabic{section}.\arabic{proposition}}

\title{A Parallel Orbital-Updating Approach for Electronic Structure Calculations
\thanks{This work was partially
supported by  the Funds for Creative Research Groups of China under grant 11321061,
the National Basic Research Program of China under grant
2011CB309703, the National Science Foundation of China
under   grants 11101416 and 91330202, the National 863 Project of China
under grant 2012AA01A309, and the National Center for Mathematics and Interdisciplinary Sciences of Chinese Academy of Sciences.}}

\author{\ Xiaoying Dai\thanks{LSEC,
Institute of Computational Mathematics and Scientific/Engineering
Computing, Academy of Mathematics and Systems Science, Chinese
Academy of Sciences, Beijing 100190, China ({daixy, azhou, jwzhu@lsec.cc.ac.cn}).}
\and Xingao Gong\thanks{ Department of
Physics, Fudan University, Shanghai 200433, China
(xggong@fudan.edu.cn).}
\and Aihui Zhou$^\dag$
\and Jinwei Zhu$^\dag$
}

\begin{document}


\maketitle

\begin{abstract}
In this paper, we propose an orbital iteration based parallel approach for electronic structure calculations. This approach is based on our understanding of the single-particle equations of independent particles that move in an effective potential. With this
 new approach, the solution of the single-particle equation is reduced to some solutions of
 independent linear algebraic systems and a small scale algebraic problem.
It is demonstrated by our numerical
experiments that this new approach is quite efficient for
full-potential calculations for a class of molecular systems.
\end{abstract}

\begin{keywords}
eigenvalue, eigenspace, electronic structure,  finite
element, full-potential calculation, parallel orbital-updating
\end{keywords}

\begin{AM}
35Q55, 65N25, 65N30, 65N50, 81Q05
\end{AM}

\pagestyle{myheadings}

\thispagestyle{plain}

\markboth {XIAOYING DAI, XINGAO GONG, AIHUI ZHOU, AND JINWEI ZHU, } {Parallel Orbital Updating}

\section{Introduction}\label{sec:intro}
\setcounter{equation}{0}

The many-body Schr{\" o}dinger equation for electronic structure is usually
intractable. In applications,  simplified and equivalent models that
are tractable are then desired and
 proposed. Among them there are  single-particle
approximations, such as Hartree-Fock type equation and Kohn-Sham
equations\cite{kaxiras03,kohn-sham65,lebris03,martin04,parr-yang-94}.

Within the framework of the single-particle approximation, the original intractable many-body
Schr{\" o}dinger equation is reduced to  a set of tractable single-particle equations of independent particles
that move in an effective potential, while the effective potential includes an external potential defined by
the nuclei or ions
and the effects of the electron-electron interactions. It is a quite good description
when effects of exchange and correlation are not crucial for describing the phenomena required.
The Kohn-Sham equation, for instance, is the group of single-particle equations for non-interacting quasi-particles,
whose density is the same as the exact density of real electrons. The various physical quantities can be expressed in terms of the single-particle orbitals. With these single-particle approximations, we observe that the orbitals may be computed individually or
in parallel.

The main philosophy behind this paper is that a fundamental algorithm must be built on simple rules. Indeed, supercomputers would otherwise be unable to cope with it efficiently.
We understand that the rules of fundamental physical theories on motions are very simple while the complexity of a system results from its specificity of initial condition, like Newtonian laws. In our approach, we try to use this kind of idea for
designing discretization schemes.
We note that an iteration process may be
viewed as (a discretized version of ) some motion.  Motivated by such a consideration and observation, in this paper, we will propose a new parallel approach for electronic structure calculations based on finite element discretizations and some simple
 iterations.
 The approach may be viewed as a single-particle orbital updating algorithm.
 It is shown by our investigation that a simple iteration
 with some observational data would be efficient for a special situation and supercomputers.
One prototype  algorithm for electronic structure
calculations based on finite element discretizations is stated as follows (see Section \ref{ParO}):
\begin{Algorithm}\label{algorithm1}
\begin{enumerate}
\item Given initial data  $(\lambda^{(0)}_i, u_i^{(0)})\in \mathbb{R}\times H^1_0(\Omega)$
with 
 $(u^{(0)}_i, u^{(0)}_j)_{\Omega}=\delta_{ij}, (i, j=1,2,\cdots,N)$, define  $\mathcal
{T}_0$ and $V_0$, and let $n=0$

\item Construct $\mathcal {T}_{n+1}$ and $V_{n+1}$ based on an  adaptive procedure  to $(\lambda^{(n)}_i, u_i^{(n)})$.

\item For $i=1,2,\cdots,N$, find $u^{(n+1/2)}_i\in V_{n+1}$ satisfying
 \begin{eqnarray*}
 a(U^{(n)}; u^{(n+1/2)}_i, v)= \lambda_i^{(n)}(u_i^{(n)},v) ~\forall v\in V_{n+1}
\end{eqnarray*}
in parallel.

\item Project to eigenspace: find $(\lambda^{(n+1)},u^{(n+1)})\in \mathbb{R}\times {\tilde V}_{n+1}$ satisfying
$\|u^{(n+1)}\|_{0,\Omega}=1$ and  \begin{eqnarray*}
a(U^{(n+1/2)}; u^{(n+1)},v)=\lambda^{(n+1)}
(u^{(n+1)},v) ~~\forall v\in {\tilde V}_{n+1}
\end{eqnarray*}
to obtain eigenpairs
$(\lambda^{(n+1)}_i,u^{(n+1)}_i)(i=1,2,\cdots,N).$

\item Let $n=n+1$ and go to Step 2.
\end{enumerate}
Here  ${\tilde V}_{n+1}=~\mbox{span}~ \{u^{(n+1/2)}_1,u^{(n+1/2)}_2,\cdots,u^{(n+1/2)}_N\}$,  $U^{(n)} = (u_1^{(n)},
u_2^{(n)},\cdots, u_N^{(n)})$, $U^{(n+1/2)} = (u_1^{(n+1/2)},u_2^{(n+1/2)},\cdots, u_N^{(n+1/2)})$,
  and $a(\cdot;\cdot,\cdot)$ is the
nonlinear variational form associated the Kohn-Sham equation defined
in Section \ref{KS-equation}.
\end{Algorithm}

We understand that modern computational science does not have an algorithm
setting up initial conditions, it can only determine initial
conditions through physical observation and data. In Step 1 of Algorithm \ref{algorithm1}, we may choose

\begin{itemize}
\item Gaussian-type orbital or Slater-type orbital based guesses, which are applicable to full-potential calculations,

\item local plane-wave discretization based guesses, which are applicable to pseudo-potential settings,

\item local finite element/volume discretization based guesses, which are applicable to either
full-potential calculations or pseudo-potential settings.
\end{itemize}
Step 2 is used to deal with the singularity of Coulomb potentials or the highly oscillating
behaviors of eigenfunctions. Step 3 is to solve some source problems while Step 4 is an eigenvalue problem of small scale.

Note that Kohn-Sham equation is a nonlinear eigenvalue problem. To handle the
nonlinearity,
 the so called self-consistent field (SCF) (see, e.g.,
\cite{kresse-furthmuller-96, payne-92}) iteration approach is usually
applied. After some discretization, the central computation of
such kind of nonlinear eigenvalue problems is the repeat computation
of the following algebraic eigenvalue problem $$A u  = \lambda B
u,$$ where $A$ is the stiff matrix, $B$ is the mass matrix. For the above
algebraic  eigenvalue problem, we need to solve the first $N$
eigenvalues and the corresponding eigenfunctions or eigenspaces.  If $A$ and $B$ are
sparse, then the optimal computational complexity is
$\mathcal{O}(N^2 N_g)$, while if $A$ or $B$ is dense, then the
optimal computational complexity is $\mathcal{O}(N N_g^2)$.
Here $N_g$ is the
dimension of the matrix. Using a finite element, finite difference
or finite volume method to discretize Kohn-Sham equation is the former
case, while using plane wave functions as the bases or use Gaussian
type bases belong to the latter case. Usually, one need to solve
tens of such algebraic eigenvalue problem, and either $N_g \gg N$ or a practically complete
basis set is difficult to obtain \cite{andersen75,blochl94,bowler-miyazaki12,losilla-sundholm12,troullier-martins91,vanderbilt90}.
So, the cost will be the lowest when $N_g$ becomes $N$.

With this new algorithm, we see that the solution of the
original tens of large scale eigenvalue problems will be reduced to
the solution of some independent source problems and some eigenvalue
problem of small scales. Consequently, the optimal computational
complexity becomes $\mathcal{O}(N_g + N^3)$, which
is much lower than either $\mathcal{O}(N^2 N_g)$ or
$\mathcal{O}(N N_g^2)$. Besides, since the $N$ source problems in Step 3 are
independent each other, they can be calculated in parallel
intrinsically. This indicates that our algorithm can use more processors and hence possesses a supercomputing potential.

The rest of this paper is organized as follows. In Section \ref{sec-preliminaries},
we provide some preliminaries for Kohn-Sham DFT problem setting.
We then propose our new parallel orbital-updating approach for electronic structure calculations
in Section \ref{sec-algorithm-ks}. In Section \ref{sec-numerical}, we present some numerical experiments
that show the efficiency of our new algorithm. Finally, we give some concluding remarks and an appendix.

\section{Preliminaries}\label{sec-preliminaries}
Let $\Omega\subset \mathbb{R}^3$ be a polyhedral domain.
We shall use the standard notation for Sobolev spaces $H^{s}(\Omega)(1\le s<\infty)$ and their
associated norms and seminorms, see, e.g., \cite{adams75}. We denote  $H^1_0(\Omega)=\{v\in H^1(\Omega): v\mid_{\partial\Omega}=0\}$, where $v\mid_{\partial\Omega}=0$ is understood in
the sense of trace and $(\cdot,\cdot)$ is the standard $L^2$ inner product.

Let $\{\mathcal{T}_h\}$ be a shape regular family of nested
conforming meshes over $\Omega$ with size $h$ that is small enough:
there exists a constant $\gamma^{\ast}$ such that
\begin{eqnarray}\label{shape-regularity}
\frac{h_{\tau}}{\rho_{\tau}} \leq \gamma^{\ast} \quad\forall~\tau \in \mathcal{T}_h,
\end{eqnarray}
where $h_{\tau}$ is the diameter of $\tau$ for each $\tau\in \mathcal{T}_h$, $\rho_{\tau}$ is the
diameter of the biggest ball contained in $\tau$,
and $h=\max\{h_{\tau}: \tau\in \mathcal{T}_h\}$. Let $\mathcal{E}_h$ denote the set of interior faces
(edges or sides) of $\mathcal{T}_h$.

Let $S^{h,k}(\Omega)$ be a subspace of continuous functions on $\Omega$ such that
\begin{eqnarray*}
S^{h,k}(\Omega)=\{v\in C(\bar{\Omega}): ~v|_{\tau}\in P^k_{\tau}
\quad\forall~\tau \in \mathcal{T}_h\},
\end{eqnarray*}
where $P^k_{\tau}$ is the space of polynomials of degree no greater than $k$ over $\tau$.
$S^{h,k}(\Omega)$  are usually called finite element spaces.
Let $S^{h,k}_0(\Omega)=S^{h,k}(\Omega)\cap H^1_0(\Omega)$.
We shall denote $S^{h,k}_0(\Omega)$ by $S^{h}_0(\Omega)$ for simplification of notation afterwards.

\subsection{Adaptive finite element approximation}\label{afea}
To handle the Coulomb potential or the highly
oscillating behaviors of eigenfunctions  efficiently, we apply an adaptive
finite element approach to discretize the associated source problems. An
adaptive finite element  algorithm usually consists of the following
loop
\cite{cascon-kreuzer-nochetto-siebert08,chen-dai-gong-he-zhou13,dai-gong-yang-zhou-2011,
dai-xu-zhou08}:
$$
\mbox{\bf Solve}~\rightarrow~\mbox{\bf Estimate}~\rightarrow~\mbox{\bf Mark}~\rightarrow~\mbox{\bf Refine}.
$$

 We shall replace the subscript $h$ (or $h_k$) by an iteration
counter $k$  whenever convenient afterwards.

{\bf Solve.}
Get the piecewise polynomial finite element approximation with respect to a given
 mesh $\mathcal{T}_k$.

{\bf Estimate.} Given a mesh $\mathcal{T}_k$ and the corresponding output   from the
``Solve'' step, ``Estimate'' presents the a posteriori error estimators
$\{\eta_k(\cdot
,\tau)\}_{\tau\in\mathcal{T}_k}$. 

{\bf Mark.}
Based on the a posteriori error indicators
$\{\eta_k(\cdot,\tau)\}_{\tau\in\mathcal{T}_k}$, ``Mark" provides a strategy to
choose a subset $\mathcal{M}_k$ of elements  of $\mathcal{T}_k$ for refining.

{\bf Refine.} Associated with the mesh $\mathcal{T}_k$ and the set of marked elements $\mathcal{M}_k$,
``Refine'' produces a new mesh $\mathcal{T}_{k+1}$ by refining all elements
in $\mathcal{M}_k$ at least one time.

One of the most widely used marking strategy to enforce error reduction is the following so-called D\"{o}rfler strategy \cite{dorfler96}.

\vskip0.1cm
{D\"{o}rfler Strategy.}\quad Given a marking parameter $0<\theta <1$ :
\begin{enumerate}
\item Choose a subset $\mathcal{M}_k \subset \mathcal{T}_k$ such that
\begin{eqnarray}\label{dorfler-strategy}
\sum_{\tau\in \mathcal{M}_k}\eta^2_k(\cdot, \tau)  \geq \theta
\sum_{\tau\in \mathcal{T}_k} \eta^2_k(\cdot,\tau).
\end{eqnarray}
\item Mark all the elements in $\mathcal{M}_k$.
\end{enumerate}

The ``Maximum Strategy" is another wildly  used marking strategy.

\vskip0.1cm
{Maximum  Strategy.}\quad Given a marking parameter $0<\theta <1$ :
\begin{enumerate}
\item Choose a subset $\mathcal{M}_k \subset \mathcal{T}_k$ such that
\begin{eqnarray}\label{maximum-strategy}
 \eta_k(\cdot, \tau)  \geq \theta
\max_{\tau\in \mathcal{T}_k}\eta_k(\cdot,\tau).
\end{eqnarray}
\item Mark all the elements in $\mathcal{M}_k$.
\end{enumerate}


In our computation, we apply the shape-regular bisection for the refinement. We refer to
\cite{chen-dai-gong-he-zhou13,dai-gong-yang-zhou-2011} for more details for the adaptive finite element computations for Kohn-Sham DFT.

\subsection{Kohn-Sham equation}\label{KS-equation}
The Kohn-Sham equation of a molecular system consisting of $M$ nuclei of charges
$\{Z_1,\cdots,Z_M\}$ located at positions $\{{\bf
R}_1,\cdots,{\bf R}_M\}$ and $N_e$ electrons in the non-relativistic
and spin-unpolarized setting is the following nonlinear eigenvalue problem
\begin{eqnarray}\label{eq-ks}
\left\{ \begin{array}{rcl} \left(-\frac{1}{2}\Delta + V_{\rm ext}+ V_H(\rho) +  V_{\rm xc}(\rho)\right)u_i &=& \lambda_i u_i\quad\mbox{in}~\mathbb{R}^3,
\\[1ex] \displaystyle
\int_{\mathbb{R}^3} u_i u_j &=& \delta_{ij}, \quad i, j=1,2,\cdots, N,
\end{array} \right.
\end{eqnarray}
where $\displaystyle \rho(x)=\sum_{i=1}^N| u_i(x)|^2$ is the electron density, $ V_H(\rho) =
\frac{1}{2} \int_{\mathbb{R}^3}
\frac{\rho(y)}{|x-y|}dy $ denotes  the Hartree  potential,
$V_{\rm xc}(\rho)$ indicates the exchange-correlation potential, and $\displaystyle V_{\rm ext}(x)$ is  the  electrostatic potential generated by the nuclei, including both full-potentials  and pseudopotential approximations. For full-potentials, $N = N_e$ and $\displaystyle V_{\rm ext}(x)=-\sum_{k=1}^M\frac{Z_k}{|x-{\bf R}_k|}$.
While for  pseudopotential approximations,  $N$ equals to the number of valence electrons, $V_{\rm ext}  = V_{\rm loc} + V_{\rm nl}$, with $V_{\rm loc}$ being the local part of pseudopotential and  $V_{\rm nl}$ being a
nonlocal part  given
by (see, e.g., \cite{martin04})
$$
V_{\rm nl}\phi=\sum_{j=1}^n(\phi,\zeta_j)\zeta_j
$$
with $ n\in\mathbb{N}$ and $\zeta_j\in L^2(\Omega) (j=1,2,\cdots,n)$.

By density functional theory (DFT)
\cite{hohenberg-kohn64,kohn-sham65}(see, also, \cite{martin04,parr-yang-94}), the ground state (charge) density
of the system may be obtained by solving the
 lowest $N$ eigenpairs of (\ref{eq-ks}).
In computation,  $V_H(\rho)$ is usually obtained by solving the following  Poisson equation:
\begin{align}
-\Delta V_H(\rho)=4\pi\rho(x).
\end{align}
The exact formula for exchange-correlation potential $V_{xc}$ is unknown.
Some approximation (such as  LDA, GGA) has to be used.

Note that the ground state wavefunction of the Schr\"{o}dinger equation and
the solutions of Kohn-Sham model are  exponentially decay \cite{agmon81,almbladh85, yesrentant10},   $\mathbb{R}^3$ is then replaced by some polyhedral  domain $\Omega\subset\mathbb{R}^3$  in computation.
Instead of (\ref{eq-ks}), more precisely, we  solve
\begin{eqnarray}\label{weak-ks}
\left\{ \begin{array}{rcl} a(U; u_i, v)  &=&
\displaystyle \lambda_{i}\big(   u_i, v \big)
\quad\forall~ v\in H_0^1(\Omega), \\[1ex]
\displaystyle \int_{\Omega}u_{i} u_{j} &=& \delta_{ij}, \quad i, j=1,2,\cdots,N,
\end{array} \right.
\end{eqnarray}
where $U = (u_1, u_2,\cdots, u_N)$ and $a(U; \cdot, \cdot): H_0^1(\Omega) \times H_0^1(\Omega) \rightarrow \mathbb{R}$ is defined by
\begin{eqnarray}\label{eq-operator-A}
~~a(U; w, v)  &=& \frac{1}{2} (\nabla w, \nabla v)  + \left((V_{\rm ext}+V_H(\rho)+V_{xc}(\rho_{_U})) w,v\right)
~~\forall~ w,v\in H_0^1(\Omega)
\end{eqnarray} with $\rho_{_U}=\displaystyle\sum^N_{i=1}|u_i|^2.$

\section{Parallel orbital-updating algorithm}\label{sec-algorithm-ks}
In this section, we shall introduce the parallel orbital-updating algorithm for the Kohn-Sham equation.
Obviously, the same idea can be applied to solve Hartree-Fock equations as well as other eigenvalue problems.

 \subsection{Finite dimensional discretization}
Consider a finite dimensional discretization of (\ref{weak-ks}) as
follows:
\begin{eqnarray}\label{fem-dis-ks}
\left\{ \begin{array}{rcl} a(U_n; u_{n, i}, v_n)  &=&
\displaystyle \big(   \lambda_{n, i}u_{n, i}, v_n \big)
\quad\forall~ v_n\in V_n, \\[1ex]
\displaystyle \int_{\Omega}u_{n, i} u_{n, j} &=& \delta_{ij}, \quad i, j=1,2,\cdots,N,
\end{array} \right.
\end{eqnarray}
where $V_n$ is some finite dimensional space. We see that  (\ref{fem-dis-ks}) is a nonlinear eigenvalue problem,
and the so called self-consistent field (SCF)(see, e.g.,
\cite{kresse-furthmuller-96, payne-92}) iteration approaches are
often used to linearize it.

We may divide the finite
dimensional discretizations for Kohn-Sham equations into
three classes: the plane wave method, the local basis set method, and the real
space method. The plane wave method uses plane wave functions as the
basis functions to span a finite dimensional space $V_n$, while the local
basis set method uses some Slater type or Gaussian type functions as
the bases to construct a finite dimensional space $V_n$.
The finite element method is one of commonly used real space methods, where
finite element bases are used to construct  $V_n$. No matter what kind of methods are used to discretize the
Kohn-Sham equation, after linearization, what we get is some
algebraic eigenvalue problem $A u  = \lambda B u$, where $A$ is the
stiff matrix, $B$ is the mass matrix. And we need to solve the first $N$ eigenvalues and the
corresponding eigenfunctions of the algebraic eigenvalue problem.  If $A$ and $B$ are sparse, e.g.,
discretized by a finite element method, then the optimal computational
complexity is $\mathcal{O}(N^2 N_g)$, while if $A$ or $B$ is dense,
e.g., discretized by a plane wave method or a local basis set method, then
the  optimal  computational complexity becomes $\mathcal{O}(N  N_g^2)$.
Here,  $N_g$ is the dimension
of the matrix. We see that  $N_g \gg N$ or a practically complete basis set is difficult to obtain. Note that to obtain an accurate approximation, one
needs to solve tens of such algebraic eigenvalue problems, which limits the application of Kohn-Sham DFT to large
scale systems, especially for full-potential calculations.

\subsection{Parallel orbital-updating approach}\label{ParO}

We have already seen that the huge computational
complexity for solving the discretized eigenvalue problems limits the
application of Kohn-Sham DFT to the electronic structure calculations
for large scale systems. Hence, a faster, accurate and efficient algorithm for solving
Kohn-Sham equations as well as other eigenvalue problems is desired.
In this subsection,
we propose some new approach that can reduce the computational
complexity remarkably, as compared with the existing methods of solving Kohn-Sham equations.

We understand that the Kohn-Sham equation is established within the
framework of single-particle approximation,
 and can be viewed as a set of single-particle equations of independent particles that move in an
 effective potential.
Motivated by the setting that independent particle moves in an efficient potential, we
are indeed able to carry out
 these single-particle orbitals  individually or in parallel intrinsically.
It is shown by our investigation that a simple
iteration with some observational data would be efficient for a
special situation and supercomputers.

We note that for solving a large scale eigenvalue problem, an
iteration scheme is usually used, which can also be view as ( a discretized version of) some
motion. Based on our understanding that the basic rule of motion is
simple but the initial data may be special, we propose the following
parallel orbital-updating algorithm for solving the Kohn-Sham
equation based on finite element discretizations.\vskip 0.2cm

\begin{Algorithm}\label{alg-orbital-updating-ks}

\begin{enumerate}
\item Given initial data  $(\lambda^{(0)}_i, u_i^{(0)})\in \mathbb{R}\times H^1_0(\Omega)$
with $\|u^{(0)}_i\|_{0,\Omega}=1 (i=1,2,\cdots,N)$, define $\mathcal
{T}_0$ and $V_0$, and let $n=0$

\item Construct $\mathcal {T}_{n+1}$ and $V_{n+1}$ based on an adaptive procedure  to $(\lambda^{(n)}_i, u_i^{(n)})$.

\item For $i=1,2,\cdots,N$, find $u^{(n+1/2)}_i\in V_{n+1}$ satisfying
 \begin{eqnarray}\label{alg-bvp}
 a(U^{(n)}; u^{(n+1/2)}_i, v)= \lambda_i^{(n)}(u_i^{(n)},v) ~\forall v\in V_{n+1}
\end{eqnarray}
in parallel.

\item Project to eigenspace: find $(\lambda^{(n+1)},u^{(n+1)})\in \mathbb{R}\times {\tilde V}_{n+1}$ satisfying
$\|u^{(n+1)}\|_{0,\Omega}=1$ and
 \begin{eqnarray}\label{alg-evp}
a(U^{(n+1/2)}; u^{(n+1)},v)=\lambda^{(n+1)}
(u^{(n+1)},v) ~~\forall v\in {\tilde V}_{n+1}
\end{eqnarray}
to obtain eigenpairs
$(\lambda^{(n+1)}_i,u^{(n+1)}_i)(i=1,2,\cdots,N).$

\item Let $n=n+1$ and go to Step 2.
\end{enumerate}
Here  ${\tilde V}_{n+1}=~\mbox{span}~ \{u^{(n+1/2)}_1,u^{(n+1/2)}_2,\cdots,u^{(n+1/2)}_N\}$, $U^{(n)} =
(u_1^{(n)}, u_2^{(n)},\cdots, u_N^{(n)})$, and $U^{(n+1/2)} =
(u_1^{(n+1/2)}, u_2^{(n+1/2)},\cdots, u_N^{(n+1/2)})$.
\end{Algorithm}

In our computation in Section \ref{sec-numerical}, we choose $V_{n}$ as some finite element space
$S_0^{h_n, k}(\Omega)$ over $\mathcal {T}_n$ with some $k$, the degree of piecewise
polynomials. While $\mathcal {T}_n$ is constructed based on some a posteriori error estimations of $(\lambda^{(n)}_i, u_i^{(n)})$, marking strategy and refine procedure as described in Section \ref{afea}. Indeed, $V_n$ can also be any other appropriate relevant finite dimensional
spaces.

We see that to provide a physical observational data is natural and
significant for algorithm's efficiency. In  electronic structure
calculations, we may choose the initial data $(\lambda^{(0)}_i,
u_i^{(0)})$  as follows:
\begin{itemize}
\item Gaussian-type orbital or Slater-type orbital based guesses, which are applicable to full-potential calculations,

\item local plane-wave discretization based guesses, which are applicable to pseudo-potential settings,

\item local finite element/volume discretization based guesses (c.f., e.g.,\cite{chen-he-zhou11,dai-gong-yang-zhou-2011,dai-shen-zhou08,daix-zhou08}), which are applicable to either
full-potential calculations or pseudo-potential settings.
\end{itemize}

It will be demonstrated by our experiments in Section \ref{sec-numerical} that the parallel orbital-updating approach is powerful in
electronic structure calculations. Here we mention  several features of Algorithm
\ref{alg-orbital-updating-ks}  as follows:\vskip 0.2cm

\begin{itemize}

\item {\bf Model linearization}. We understand from Algorithm \ref{alg-orbital-updating-ks} that our parallel orbital-updating
approach is in fact some new SCF iteration technique and mixes
simple discretization iterations of source problems with adaptive
computations.\vskip 0.2cm

\item {\bf Complexity reduction}. We see form our new algorithm that the solution
of the original tens of large scale eigenvalue problems is reduced
to the solution of some independent source problems and some small
scale eigenvalue problems. Since these source problems are
independent each other, our parallel orbital-updating  algorithm can be carried out
in parallel intrinsically. Indeed, the computational complexity then becomes
 $\mathcal{O}(N_g + N_{orb}^3)$, which is much lower than
$\mathcal{O}(N_{orb}^2 N_g)$ or $\mathcal{O}(N_{orb} N_g^2)$, the
costs for  solving the Kohn-Sham equation directly, where $N_{orb}$ is the number of desired
eigenvalues and $N_g$ is the number of unknowns of
the discretized Kohn-Sham equation or the
dimension of the resulting matrix. \vskip 0.2cm

\item {\bf Eigenvalue separation}. If the initial guess is well-posed, then we are able to obtain the orbital approximations individually.
For illustration,  we comment the more general setting Algorithm
\ref{basic-algorithm} rather than  Algorithm \ref{alg-orbital-updating-ks}, which is applied to solve eigenvalue problem (\ref{eigen}).


 For $\lambda\in \sigma(L)$, define
$$
\delta(\lambda)=\inf\{\|w-v\|_{l,\Omega}: w\in M(\lambda), v\in M(\mu), \mu\in \sigma(L)\setminus\{\lambda\}\}.
$$
Let $\mu_i\in\sigma(L)$ whose multiplicity is $q_i(i=1,2,\cdots,M)$, respectively\footnote{In our discussion, $q_1=1$ when $\mu_1=\lambda_1$.}.
Assume that the
initial data $(\lambda^{(0)}_{i+j}, u_{i+j}^{(0)})\in \mathbb{R}\times
H^1_0(\Omega)$
satisfying that  $\|u^{(0)}_{i+j}\|_{0,\Omega}=1 (i=1,2,\cdots,M; j=0,1,2,\cdots,q_i-1)$ and $\mbox{span}\{u_{i+j}^{(0)}:j=0,1,2,\cdots,q_i-1\}$ away from
$$\cup \left\{M(\mu): \mu\in \sigma(L)\setminus\{\mu_i\}\right\},~ i=1,2,\cdots,M.$$ More precisely, for
$i=1,2,\cdots,M$,
$$
\inf\{\|u^{(0)}_{i+j}-v\|_{0,\Omega}: j=0,1,2,\cdots,q_i-1; v\in M(\mu), \mu\in \sigma(L)\setminus\{\mu_i\}\}>\delta(\mu_i)/2.
$$
Then
we may apply  Algorithm \ref{basic-algorithm} to obtain  eigenspace approximations in parallel. We observe that the computational cost $$\mathcal{O}\left(\left(\sum_{i=1}^M q_i\right)^3\right)$$ of computing eigenpairs is then reduced to
$$\mathcal{O}\left(\sum_{i=1}^Mq_i^3\right).$$
\vskip 0.2cm

\item {\bf Two-level parallelization}. We may also see a potential of our algorithm in
large scale parallel computation.  Since these source problems are
independent each other, they can be calculated in parallel
intrinsically, and  each source problems can also be solved in
parallel by various multigrid or domain decomposition approaches,
which is a two level parallelization. One level is the
 solution of these $N$ independent source problems in parallel intrinsically,
another level is to solve each source problem in parallel by using
idea of multigrid or domain decomposition methods. As a result,
there may be some potential for E-scale eigenvalue computations.
\end{itemize}\vskip 0.2cm


We point out that if the initial guess is not well-provided, then we suggest to apply the following algorithm:

\begin{Algorithm}\label{alg-orbital-updating-ks-2}

\begin{enumerate}
\item Given initial data  $(\lambda^{(0)}_i, u_i^{(0)})\in \mathbb{R}\times H^1_0(\Omega)$
with $\|u^{(0)}_i\|_{0,\Omega}=1 (i=1,2,\cdots,N + m)$, define $\mathcal
{T}_0$ and $V_0$, and let $n=0$

\item Construct $\mathcal {T}_{n+1}$ and $V_{n+1}$ based on an adaptive procedure  to $(\lambda^{(n)}_i, u_i^{(n)})$.

\item For $i=1,2,\cdots,N + m$, find $u^{(n+1/2)}_i\in V_{n+1}$ satisfying
 \begin{eqnarray*}
 a(U^{(n)}; u^{(n+1/2)}_i, v)= \lambda_i^{(n)}(u_i^{(n)},v) ~\forall v\in V_{n+1}
\end{eqnarray*}
in parallel.
\item Project to eigenspace: find $(\lambda^{(n+1)},u^{(n+1)})\in \mathbb{R}\times {\tilde V}_{n+1}$ satisfying
$\|u^{(n+1)}\|_{0,\Omega}=1$ and  \begin{eqnarray*}
a(U^{(n+1)}; u^{(n+1)},v)=\lambda^{(n+1)}
(u^{(n+1)},v) ~~\forall v\in {\tilde V}_{n+1}
\end{eqnarray*}
to obtain eigenpairs
$(\lambda^{(n+1)}_i,u^{(n+1)}_i)(i=1,2,\cdots,N + m).$

\item Let $n=n+1$ and go to Step 2.
\end{enumerate}
Here  ${\tilde V}_{n+1}=~\mbox{span}~ \{u^{(n+1/2)}_1,u^{(n+1/2)}_2,\cdots,u^{(n+1/2)}_{N + m}\}$, $U^{(n)} =
(u_1^{(n)}, u_2^{(n)},\cdots, u_N^{(n)})$,  $U^{(n+1/2)} =
(u_1^{(n+1/2)}, u_2^{(n+1/2)},\cdots, u_N^{(n+1/2)})$, and $m$ is some proper integer.
\end{Algorithm}

It tells from Algorithm \ref{alg-orbital-updating-ks-2} that more eigenpairs should be computed so as to get
a better approximation of the first $N$ orbitals. In practice, we see that  Algorithm \ref{alg-orbital-updating-ks-2}
is more stable than  Algorithm \ref{alg-orbital-updating-ks}.

We may refer to Appendix  for a generalization to a class of linear
eigenvalue problems and its basic numerical analysis.

\section{Numerical experiments}\label{sec-numerical}

In this section, we apply our parallel orbital-updating algorithm based on finite element
discretizations to simulate several typical
molecular systems:  $H_2 O$( water), $C_9H_8O_4$(aspirin), $C_5H_9O_2
N$($\alpha$ amino acid), $C_{17}H_{19}N_3$(mirtazapine),
$C_{20}H_{14}N_4$(porphyrin), and $C_{60}$(fullerene), to show the reliability and
efficiency of our approach. Due to the
length limitation of the paper, we only show the results for full potential calculations
 for illustration.


 We understand that Gaussian\cite{gaussian} is  a popular and widely used electronic structure package. In Gaussian09,  many different basis sets with different level of
accuracy are provided for the DFT approach, for example, STO-3G, STO-6G,
6-31G,  cc-pVDZ, cc-pVTZ, cc-pVQZ, cc-pV5Z, cc-pV6Z, and  others.
The better the accuracy, the larger the cpu time cost.

To show the efficiency of our algorithm, we compare our results with
the results obtained by Gaussian09  \cite{gaussian} within the LDA81 DFT setting \cite{perdew-zunger-81}
and using bases
6-31G, cc-pVQZ, cc-pV5Z as well as cc-pV6Z, respectively. In our computation, we use also  LDA81 as the exchange
correlation functional and apply the STO-3G basis to obtain the initial guess
in  Algorithm \ref{alg-orbital-updating-ks}.

Both the results obtained  by our algorithm and those obtained by Gaussian09  are carried out on LSSC-III in
the State Key Laboratory of Scientific and Engineering Computing, Chinese Academy of Sciences, and our package RealSPACES (Real Space Parallel Adaptive Calculation of Electronic Structure)  that are based on the toolbox PHG \cite{phg} of the State Key Laboratory of Scientific and Engineering Computing, Chinese Academy of Sciences. All results are given by atomic units (a.u.).

In our tables, we use some abbreviations:
  \begin{itemize}
\item ParO =Parallel orbital-updating approach \vskip 0.2cm

\item MeshG =Mesh generation  \vskip 0.2cm

\item  SourceS =Source problem solution \vskip 0.2cm


\item DOFs=Degrees of freedom\vskip 0.2cm

\item $N_{procs}$=Number of processors
\end{itemize}\vskip 0.2cm

\subsection{Validation of reliability}

First, we will validate the reliability of our algorithm by using a
small molecular system as example.
As well known, for small atomic or molecular systems, Gaussian09 provides very efficient and accurate results
To illustrate the reliability of our algorithm, we use $H_2 O$ as an example to show that
the  results obtained by our algorithm are reliable by comparing  with those obtained by Gaussian09.

{\bf Example 1:} $H_2O$

The atomic configuration for molecule  $H_2 O$ is figured in Fig.
\ref{cfg-H2O}. Here we compute the first $5$ eigenpairs of the
associated Kohn-Sham equation.

\begin{figure}[htbp]
\centering
\includegraphics[height=3.5cm]{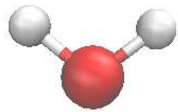}
\caption{Configuration of $H_2 O$($N_{orb}=5$)}\label{cfg-H2O}
\end{figure}

Table \ref{table-H2O} provides the detailed information for results obtained by Gaussian09  with different kinds of basis functions being used and those obtained by our algorithm after different number of iterations. The second column is the total energy obtained, the third column is the number of basis functions used, the forth column is the wall-time cost, while the fifth column is the number of processors used.

It should be pointed out that in our current numerical
experiments, the $N_{orb}$ number of boundary value problems are
solved one by one, not in parallel, and here the multi processors
deal with only the parallelization in space, not in orbital. That is, the $N_{orb}$ boundary value problems $(\ref{alg-bvp})$
are solved one by one, not in parallel. If we have enough processors, the
 $N_{orb}$ boundary value problems can be carried out in parallel intrinsically, the wall-time cost by our
algorithm  would be reduced further, and it is listed in the last column with color red.

\begin{table}[ht]
\centering
\begin{tabular}{|r|r|r|r|r|r|}\hline
       Method     &$E_{tot}$(a.u.)&    DOFs    & Time(s)  &  $N_{procs}$ & \textcolor{red}{Time(s)}\\\hline
       STO-3G     &  $-74.729534 $ &   $ 7 $    & $ 1  $   &     $1 $       &           --            \\\hline
       STO-6G      &  $-75.446932 $ &   $  7 $    & $ 2  $   &     $1 $       &           --            \\\hline
       6-31G    &  $-75.814098 $ &   $  13 $    & $ 2 $  &    $ 1 $      &           --            \\\hline
       cc-pVDZ    &  $-75.850750 $ &   $ 24  $   & $ 3$  &    $ 1$        &           --            \\\hline
       cc-pVTZ    &  $-75.894284 $ &   $ 58  $    & $ 3 $  &    $ 1 $      &           --            \\\hline
       cc-pVQZ    &  $-75.904123 $ &   $ 115 $    & $ 5 $  &    $ 1 $      &           --            \\\hline
       cc-pV5Z    &  $-75.908214 $ &   $ 201 $   & $ 23$  &    $ 1$        &           --            \\\hline
       cc-pV6Z    &  $-75.909024 $ &   $  322 $   & $ 1111$  &    $ 1$        &           --            \\\hline
       \hline
       ParO ($116$)          & {\color{blue} $-75.908257$ }   &  $594027 $ &  {\color{blue} $1096.8 $        }  &       $32$     & \textcolor{red}{$939.5$} \\\hline
       ParO ($149$)          & {\color{blue} $ -75.909026  $ }   &  $1253803 $ &  {\color{blue} $2863.6 $        }  &       $32$     & \textcolor{red}{$2399.4$} \\\hline
       ParO ($166$)          & {\color{blue} $-75.909175  $ }   &  $2007831 $ &  {\color{blue} $4654.6 $        }  &       $32$     & \textcolor{red}{$3845.4$} \\\hline
\end{tabular}
\caption{Results of $H_2O$}\label{table-H2O}
\end{table}

We observe from Table \ref{table-H2O} that for Gaussian09, as the bases changes from STO-3G to cc-pV6Z, the total ground state energy approximation reduces from large to small, while the corresponding cpu-time cost increases fastly, especially when the ground state energy approximation is closed to the exact one. For example, when using bases from cc-pV5Z to cc-pV6Z, the total energy decreases about $0.00081$ a.u., while the cpu-time cost increases 1088s. More precisely, for the basis set method, the cost is not expensive if only less accurate result is required, while the cost will become very huge if the the accuracy of approximation increases a little bit after some critical accuracy.

We see that the ground state energy approximations obtained by our algorithm after $116$ iterations( here, $1$ iteration means doing Step 2, Step 3, and Step 4 of Algorithm  \ref{alg-orbital-updating-ks} once)
is $-75.908257$ a.u., which is very close to that obtained by Gaussian09 using basis cc-pV5Z, and if we refine the mesh adaptively and do one iteration again , then the ground state energy approximation will decrease further. Let us take a more detailed. After $149$ iterations, the energy approximation decreases to $ -75.909026$ a.u., close to the results obtained by Gaussian09 using basis cc-pV6Z. If we do more iterations, that is, after $166$ iterations, the energy approximation decreases to $-75.909175$ a.u., smaller than that obtained by  Gaussian09 using basis  cc-pV6Z by $0.000151$ a.u..

Although the total cpu time cost by our algorithm is much longer than that cost by Gaussian09, we should note that the cpu-time cost does not increase as quickly as that for Gaussian09. In fact, for Gaussian09,
 when the bases are chosen from cc-pV5Z to cc-pV6Z, the total energy approximation decreases from $-75.908214$ a.u.  to $-75.909024$ a.u.,  the cpu-time cost increases from $23$s to  $1111$s, about $50$ times of the former one. For our algorithm, from $116$ iterations to $149$ iterations, the energy approximation  decreases from $-75.908257$ a.u. to $-75.909026$ a.u., the cpu-time cost increases from $1096.8$s to $2863.6$s, only about $3$ times of the former one.

Fig.  \ref{fig-H2O} shows the convergence curve for the ground state energy approximations over each adaptive refined mesh, where the x-axis is the DOFs, and the y-axis is the ground state energy approximation. We see from Fig.  \ref{fig-H2O} that the total energy approximation converges as the
number of degrees increases.  For comparison, we show the  convergence curve of the total energy obtained by
Gaussian09 with different kinds of bases(STO-3G, STO-6G, 6-31G, cc-pVDZ, cc-pVTZ, cc-pVQZ, cc-pV5Z, cc-pV6Z from the right to the left) in Fig. \ref{fig-H2O-g09}.

\begin{figure}[htbp]
\includegraphics[width = 8cm,angle=-90]{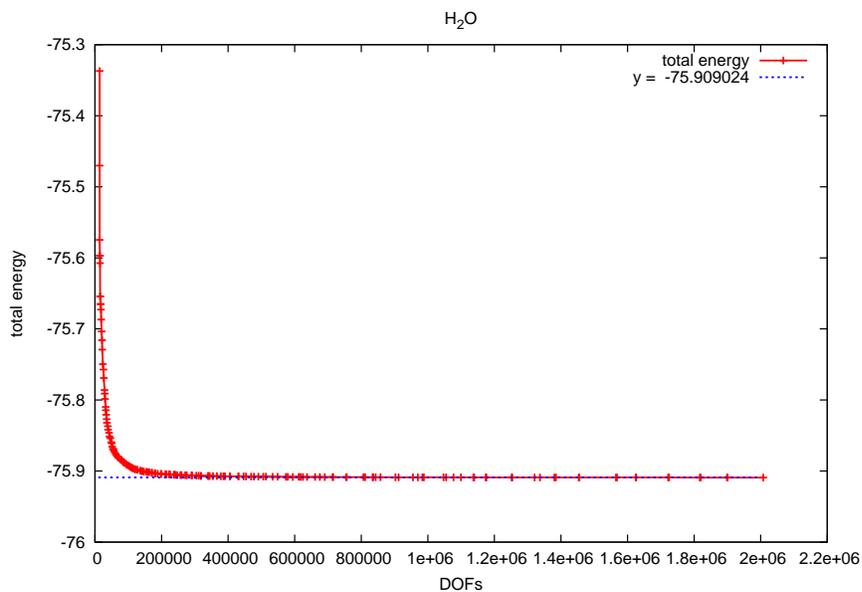} ~~
\caption{Convergence curves of the ground state energy}\label{fig-H2O}
\end{figure}

\begin{figure}[htbp]
\begin{center}
\includegraphics[width = 8cm,angle=-90]{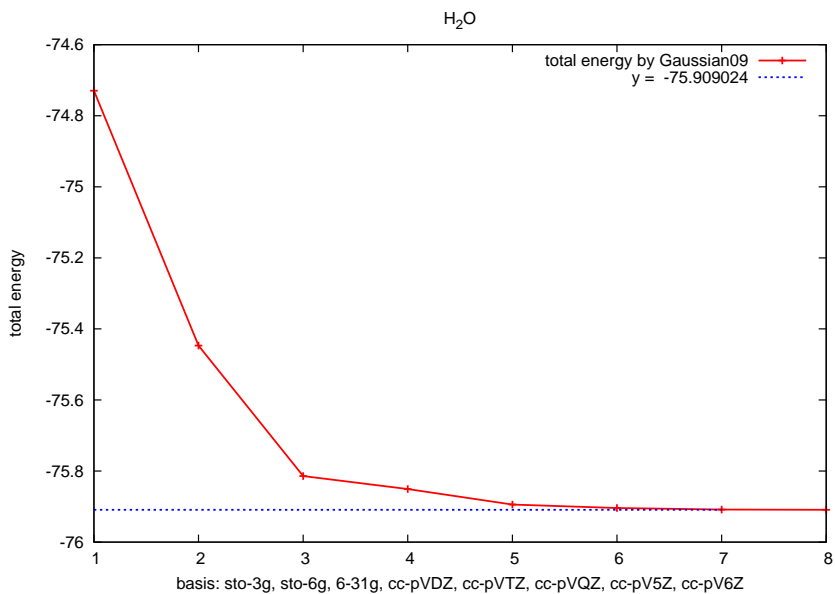} ~~
\end{center}
\caption{Convergence curves of  ground state energy by Gaussian09}\label{fig-H2O-g09}
\end{figure}

We may conclude that our algorithm can produce approximations  as accurate as or even more accurate than Gaussian09. Although for such a small example, the cpu-time cost by our algorithm is much longer than that cost by Gaussian09, we can see the potential of our algorithm when  highly accurate results are desired.

\subsection{Validation of  efficiency}

Although Gaussian09 can provide highly accurate results quickly for
small molecular systems, it is another story for molecular systems of medium or
large size. We see that the memory required for DFT method scales as $N^4$ with $N$ being
the number of basis functions, for instance.

In this subsection, we will use some molecular systems of medium
scale or large scale to show the efficiency of our algorithm.

{\bf Example 2:} $\alpha$-amino acid: $C_5H_9O_2N$

The atomic configuration for molecule  $C_5H_9O_2N$ is shown in Fig.
\ref{cfg-C5H9O2N}. For this example, we compute the first
$31$ eigenpairs of the Kohn-Sham equation, and $32$ processors are used for both Gaussian09 and our code. Table \ref{table-C5H9O2N} displays the relevant results.

\begin{figure}[htbp]
\centering
\includegraphics[height=3.5cm]{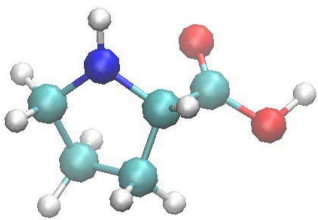}
\caption{Configuration of $C_{5}H_{9}O_2N$($N_{orb}=31$)}\label{cfg-C5H9O2N}
\end{figure}\vskip 0.2cm

\begin{table}[ht]
\centering
\begin{tabular}{|r|r|r|r|r|r|}\hline
       Method     &$E_{tot}$(a.u.)&    DOFs    & Time(s)  &  $N_{procs}$ & \textcolor{red}{Time(s)}\\\hline
       STO-3G     &  $-392.621189 $ &   $ 49  $    & $ 2.9  $   &     $32 $       &           --            \\\hline
       6-31G      &  $-397.673071 $ &   $ 90  $    & $ 3.3  $   &     $32 $       &           --            \\\hline
       cc-pVQZ    &  $-397.998732 $ &   $ 710 $    & $ 386.7 $  &    $ 32 $      &           --            \\\hline
       cc-pV5Z    &  $-398.009882 $ &   $ 1223 $   & $ 2444.6$  &    $ 32$        &           --            \\\hline
       \hline
       ParO ($127$)          & {\color{blue} $-398.009800 $ }   &  $3605766 $ &  {\color{blue} $19771.6$        }  &       $32$     & \textcolor{red}{$8708.8$} \\\hline
       ParO ($136$)          & {\color{blue} $-398.010611$ }   &  $4634263$ &  {\color{blue} $28868.7 $        }  &       $32$     & \textcolor{red}{$12101.5$} \\\hline
        MeshG      &  $-392.630620$    & $ 104037 $   & $4.5$
  &            &           $4.5$      \\ \hline
        SourceS    &  $-398.010611$          &  $4634263$    &  $17131.7$      &    $32$    & \textcolor{red}{ 364.5}   \\
\hline
\end{tabular}
\caption{Results of $C_{5}H_{9}O_2N$}\label{table-C5H9O2N}
\end{table}

We see from Table \ref{table-C5H9O2N}  that the total energy approximation obtained by Gaussian09 decreases when the bases are chosen from STO-3G to cc-pV5Z, while the cpu-time cost increases quickly. We mention that due to the huge storage requirement, we are not able to produce the result by Gassian09 using bases cc-pV6Z.

We also observe from Table \ref{table-C5H9O2N} that after $127$ iterations, the total ground state energy approximations obtained by our algorithm is $-398.009800$a.u., which is close to that obtained by cc-pV5Z. If we refine the mesh adaptively again, we can obtain more accurate results. For example, after $136$ iterations, the total energy approximation decreases to $-398.010611$a.u., which is  $0.000729$a.u. smaller than that obtained by  Gassian09 using bases cc-pV5Z. The cpu-time cost after $127$ iterations and $136$ iterations are about $19771.6$s and $28868.7$s respectively. As we addressed in  {\bf Example 1}, in our current computations, the $N_{orb}$ number of boundary value problems are solved one by one, not in parallel, and here the $32$ processors
deal with only the parallelization in space, not in orbital. Since
the $31$ boundary value problems can be carried out in parallel intrinsically,  the wall-time cost by our
algorithm cost would be at least reduced to $8708.8$s and $12101.5$s respectively provided we have
$992(32 \times 31)$ processors.

  We note from Fig. \ref{fig-C5H9O2N-energy}  that the total energy approximation converges as the
iteration increases.

\begin{figure}[htbp]
\begin{center}
\includegraphics[width = 8cm,angle=-90]{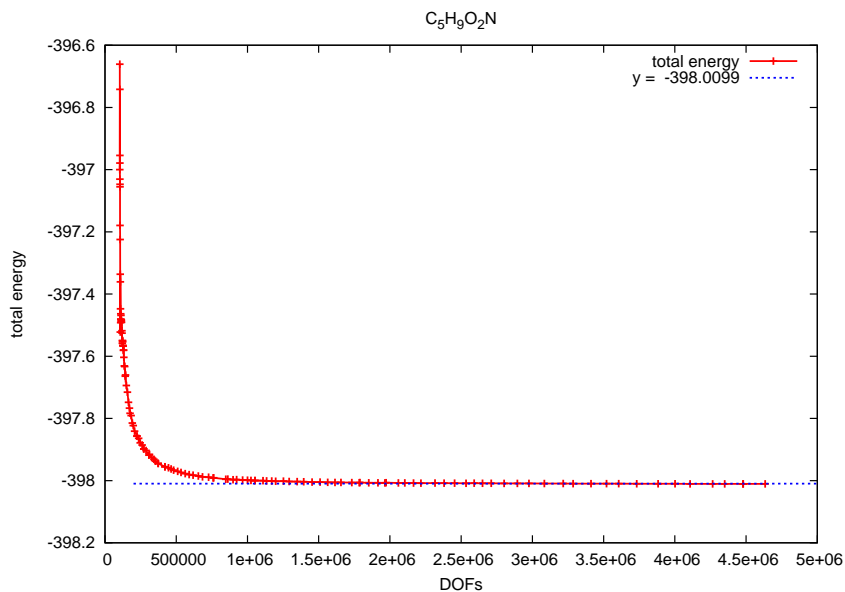} ~~
\end{center}
\caption{Convergence curves of the ground state energy}\label{fig-C5H9O2N-energy}
\end{figure}\vskip 0.2cm

{\bf Example 3:}  $C_9H_8O_4$

The atomic configuration for molecule  $C_9H_8O_4$ is shown in Fig.
\ref{cfg-C9H8O4}. We solve  the first $47$
eigenpairs of the Kohn-Sham equation.

\begin{figure}[htbp]
\centering
\includegraphics[height=3.5cm]{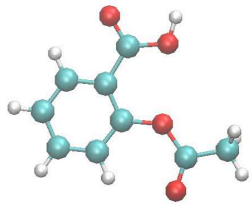}\vskip -0.4cm
\caption{Configuration of $C_9H_8O_4$($N_{orb}=47$)}\label{cfg-C9H8O4}
\end{figure}

Table
\ref{table-C9H8O4} digitizes some relevant results, including those obtained by Gaussian09 and those obtained by our algorithm over the mesh got by $79$ adaptive refinements.
Fig. \ref{fig-C9H8O4-energy}  displays the convergence curve of the ground
state energy approximation over each adaptive refined mesh.

\begin{table}[ht]
\centering \vskip -0.3cm
\begin{tabular}{|r|r|r|r|r|r|}\hline
       Method     &$E_{tot}$(a.u.)&    DOFs    & Time(s)  &  $N_{procs}$ & \textcolor{red}{Time(s)}\\\hline
       STO-3G     &  $-634.903702$  &    $73$      &  $4.2$     &     $32$       &           --            \\\hline
       6-31G      &  $-643.159782$  &    $133$     &  $10.7$    &     $32$       &           --            \\\hline
       cc-pVQZ    &  $-643.656111$  &    $955$     &  $895.8$   &    $ 32$       &           --            \\\hline
       cc-pV5Z    &  $-643.673930$  &   $1623$    &  $3782.8$  &     $32$       &           --            \\\hline
       \hline
       ParO ($91$)          & {\color{blue} $-643.648898 $ }   &  $10983675$ &  {\color{blue} $41390.8$        }  &       $32$     & \textcolor{red}{$23229.7$} \\
       \hline
        MeshG      &  $-635.473477$    & $ 609169  $   & $48.1$
  &            &           $48.1$      \\
        SourceS    &  $-643.648898$          &  $10983675$    &  $18555.9$      &    $32$    & \textcolor{red}{394.8}   \\
\hline
\end{tabular}
\caption{Results of  $C_9H_8O_4$}\label{table-C9H8O4}
\end{table}

  \begin{figure}[htbp]
\begin{center}
\includegraphics[width = 8cm,angle=-90]{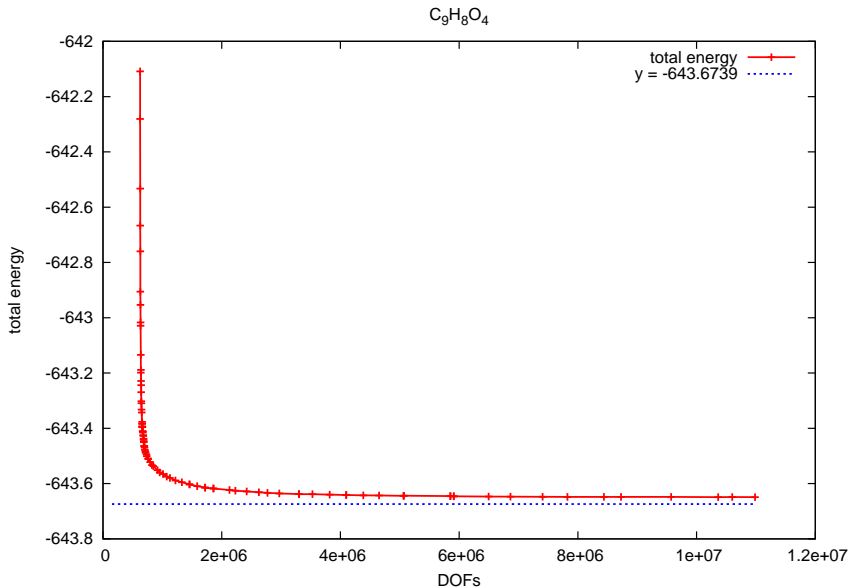} ~~
\end{center}
\caption{Convergence curves of the ground state energy}\label{fig-C9H8O4-energy}
\end{figure}\vskip 0.2cm

For this example, the total energy obtained by our algorithm are a little larger than that obtained by Gaussian09 with cc-pV5Z being used.

\subsection{Full-potential calculation for large scale system}
In this subsection, we apply our algorithm to full-potential
electronic structure calculations for some large molecular systems.

{\bf Example 4}: $C_{17}H_{19}N_3$

The  configuration for molecule  $C_{17}H_{19}N_3$ is shown in Fig.
\ref{cfg-C17H19N3} and the first
$71$ eigenpairs of the Kohn-Sham equation are computed, and $32$ processors are used for both Gaussian09 and our code.
Table
\ref{table-C17H19N3} provides the relevant results.

\begin{figure}[htbp]
\centering
\includegraphics[height=3.5cm]{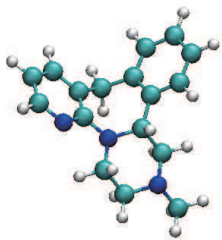}
\caption{Configuration of $C_{17}H_{19}N_3$($N_{orb}=71$)}\label{cfg-C17H19N3}
\end{figure}

\begin{table}[ht]
\centering  \vskip -0.3cm
\begin{tabular}{|r|r|r|r|r|r|}\hline
       Method     &$E_{tot}$(a.u.)&    DOFs    & Time(s)    &  $N_{procs}$  & \textcolor{red}{Time(s)} \\\hline
       STO-3G     &  $-805.972579$  &    $119$     & $4.3$        &       $32$     &            --            \\\hline
       6-31G      &  $-815.869313$  &    $218$     & $19.9$       &       $32$      &            --            \\\hline
       cc-pVQZ    &  $-816.445414$  &    $1670$    & $3456.6$     &       $32$      &            --            \\\hline
       cc-pV5Z    &  $-816.467053$  &    $2865$    & $20675.6$    &       $32$      &            --            \\\hline
        ParO ($73$)          & {\color{blue} $ -816.468992 $ }   &  $6007146$ &  {\color{blue} $ 33416.7$        }  &       $32$     & \textcolor{red}{$13885.0$} \\\hline
       ParO ($91$)          & {\color{blue} $-816.493475  $ }   &  $9963665 $ &  {\color{blue} $86465.1$        }  &       $32$     & \textcolor{red}{$31137.4$} \\\hline
       \hline
        MeshG      &  $-805.778326$    & $ 264415  $   & $31.2$
  &            &           $31.2$      \\
        SourceS    &  $-816.493475$          &  $9963665$    &  $56118.1 $      &    $32$    & \textcolor{red}{790.4}   \\
\hline
\end{tabular}
\caption{Results of $C_{17}H_{19}N_3$}\label{table-C17H19N3}
\end{table}\vskip 0.2cm

We observe from Table \ref{table-C17H19N3} that after $73$ iterations,  the total energy
 obtained by our algorithm is very close to that obtained by Gaussian09 with using bases cc-pV5Z. Let us  take a look at the cpu-time cost. We see that when using the same number of processors, the cpu-time cost by Gaussian09 and our algorithm are $20675.6$s and $33416.7$s, respectively. Note that the $N_{orb}$ boundary value problems are solved one by one, not in parallel. If we apply $2272(=32\times 71)$ processors, indeed, the cpu-time cost by our algorithm can then be reduced to  $13885.0$s. If we refine the mesh again, the energy approximation will reduce further. For instance, after $91$ iterations,
 the total energy approximation will reduce to $-816.493475$, which is $0.026422$ a.u. smaller than that obtained by Gaussian09 with cc-pV5Z being used.

Fig. \ref{fig-C17H19N3-energy}  shows the convergence curve of the ground
state energy approximation, from which we conclude that the approximation results obtained by our algorithm  converge.

  \begin{figure}[htbp]
\begin{center}
\includegraphics[width = 8cm,angle=-90]{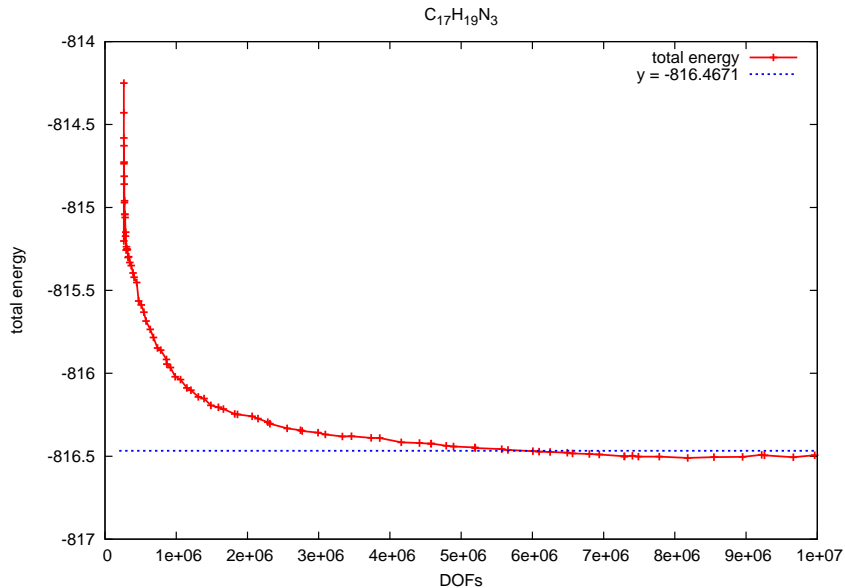} ~~
\end{center}
\caption{Convergence curves of the ground state energy}\label{fig-C17H19N3-energy}
\end{figure}\vskip 0.2cm

{\bf Example 5}: $C_{20}H_{14}N_4$

The atomic configuration for molecule  $C_{20}H_{14}N_4$ is shown in
Fig. \ref{cfg-C20H14N4} and the
first $81$ eigenpairs of the Kohn-Sham equation are approximated.

\begin{figure}[htbp]
\centering
\includegraphics[height=3.5cm]{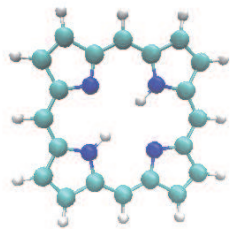}
\caption{Configuration of $C_{20}H_{14}N_4$($N_{orb}=81$)}\label{cfg-C20H14N4}
\end{figure}

Table
\ref{table-C20H14N4} digitizes the computational results, including those obtained by Gaussian09 and those obtained by our algorithm after $58$ and $68$
iterations.
Fig. \ref{fig-C20H14N4-energy}  shows the convergence curve of the ground
state energy.
¡¡
We get the similar conclusions as that obtained of {\bf Example 4} from Table \ref{table-C20H14N4}  and Fig.\ref{fig-C20H14N4-energy}.

\begin{table}[ht]
\centering \vskip -0.3cm
\begin{tabular}{|r|r|r|r|r|r|}\hline
       Method         &$E_{tot}$(a.u.)&    DOFs    & Time(s)    &  $N_{procs}$ & \textcolor{red}{Time(s)}  \\\hline
       STO-3G         &  $-968.459221$  &    $134$     &   $3.9$      &       $32 $    &              --           \\\hline
       6-31G          &  $-980.529009$  &   $ 244$     &   $12.4$     &       $32$     &              --           \\\hline
       cc-pVQZ        &  $-981.230600$  &    $1710$    &   $1798.7$   &       $32$     &              --           \\\hline
       cc-pV5Z        &  $-981.257126$  &    $2954$    &   $10499.2$  &       $32$     &              --           \\\hline
 ParO ($58$)          & {\color{blue} $ -981.257001 $ }   &  $5861789  $ &  {\color{blue} $30109.6$        }  &       $32$     & \textcolor{red}{$11738.6$} \\\hline
 ParO ($68$)          & {\color{blue} $ -981.270015  $ }   &  $10219345 $ &  {\color{blue} $77275.6$        }  &       $32$     & \textcolor{red}{$25223.8$} \\\hline
        MeshG      &  $-965.172827$    & $596951  $   & $40.9$
  &            &           $40.9$      \\
        SourceS    &  $-981.270015$          &  $10219345 $    &  $52702.5 $      &    $32$    & \textcolor{red}{650.6}   \\
\hline
\end{tabular}
\caption{Results  of $C_{20}H_{14}N_4$}\label{table-C20H14N4}
\end{table}\vskip 0.2cm

%
%

  \begin{figure}[htbp]
\begin{center}
\includegraphics[width = 8cm,angle=-90]{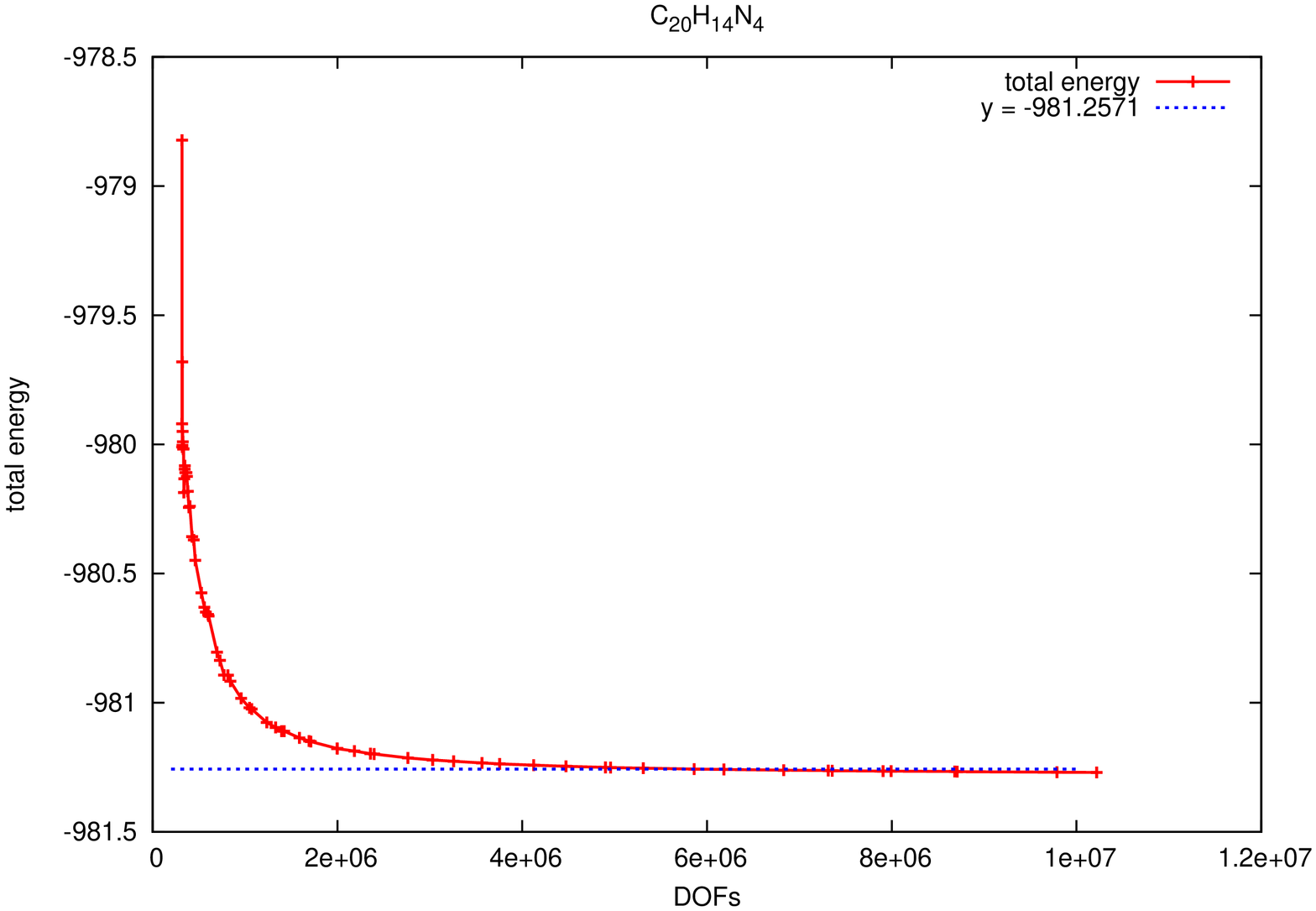} ~~
\end{center}
\caption{Convergence curves of the ground state energy}\label{fig-C20H14N4-energy}
\end{figure}
\vskip 0.2cm

{\bf Example 6}: $C_{60}$

The atomic configuration for molecule  $C_{60}$ is shown in Fig.
\ref{cfg-C60} and we compute the first $180$
eigenpairs of the Kohn-Sham equation.

\begin{figure}[htbp]
\centering
\includegraphics[height=3.5cm]{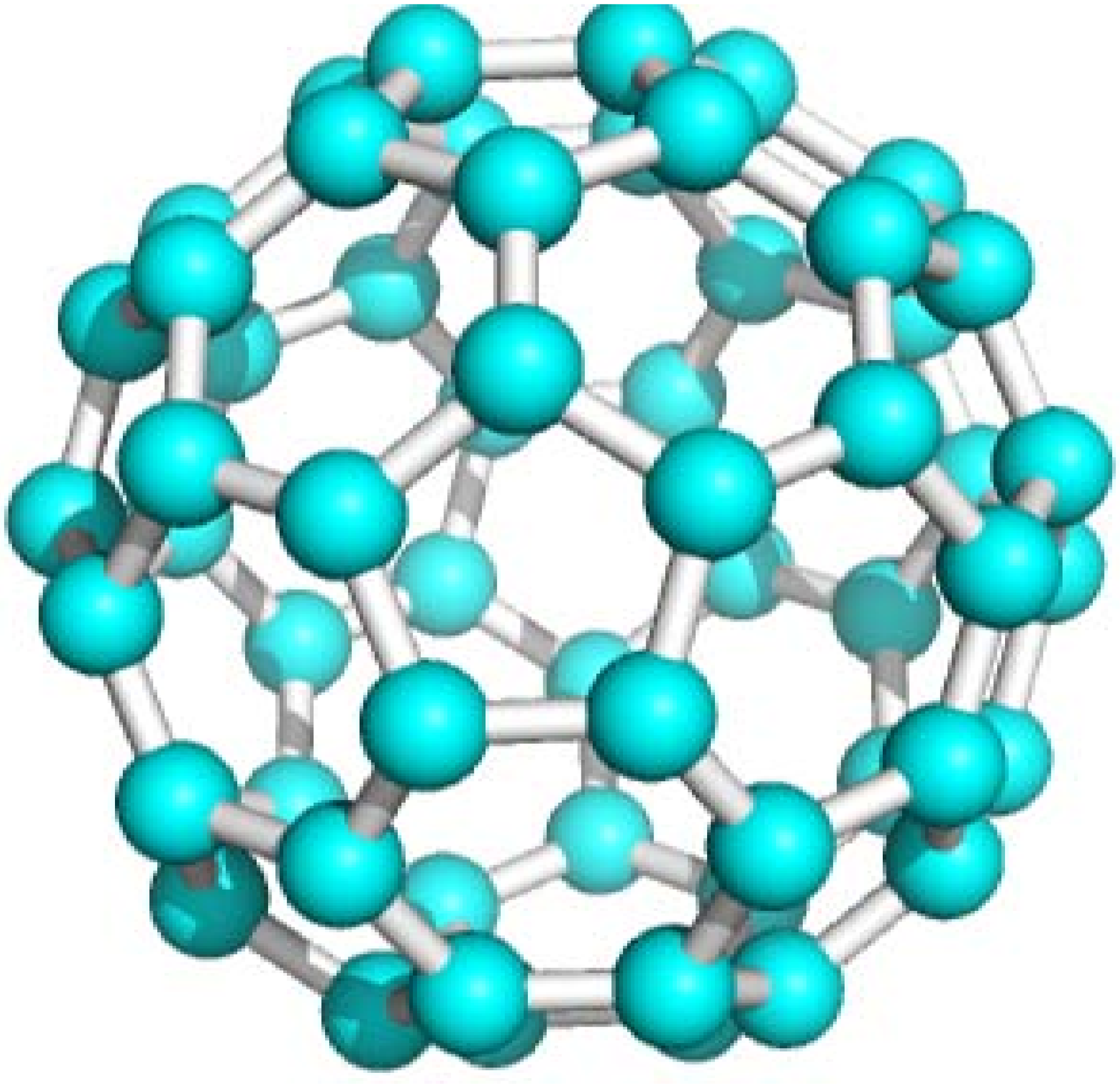}
\caption{Configuration of $C_{60}$($N_{orb}=180$)}\label{cfg-C60}
\end{figure}

Table \ref{table-C60} digitizes the relevant data. ¡¡
Fig. \ref{fig-C60-energy} provides the convergence curves of the total energy after each
iteration.

We conclude from Table \ref{table-C60} and Fig. \ref{fig-C60-energy} that  the similar conclusions as that of {\bf Example 4} can be obtained.

\begin{table}[htbp]
\centering \vskip -0.3cm
\begin{tabular}{|r|r|r|r|r|r|}\hline
       Method         &$E_{tot}$(a.u.) &    DOFs   & Time(s)     &  $N_{procs}$ & \textcolor{red}{Time(s)}   \\\hline
       STO-3G         &  $-2238.127540$  &    $300$    &   $80.5$      &      $64$      &              --            \\\hline
       6-31G          &  $-2265.319422$  &    $540$    &   $102.1$     &      $64$      &              --            \\\hline
       cc-pVQZ        &  $-2266.766797$  &   $3300$    &   $6250.3$    &      $64$      &              --            \\\hline
  cc-pV5Z        &  $-2266.824090$  &   $5460$    &   $53034.7$    &      $64$      &              --            \\ \hline \hline
  ParO ($88$)          & {\color{blue} $-2266.820160$ }   &  $10742296  $ &  {\color{blue} $88894.4$        }  &       $64$     & \textcolor{red}{$38533.2$} \\\hline
       ParO ($95$)          & {\color{blue} $-2266.837443 $ }   &  $13425339$ &  {\color{blue} $137958.4$        }  &       $64$     & \textcolor{red}{$55662.4$} \\\hline
        MeshG      &  $-2239.762286$    & $596951  $   & $174.5$
  &            &           $174.5$      \\
        SourceS    &  $-2266.837443  $          &  $13425339 $    &  $51188.3$      &    $64$    & \textcolor{red}{410.3}   \\
\hline
\end{tabular}
\caption{Results of $C_{60}$}\label{table-C60}
\end{table}\vskip 0.2cm

  \begin{figure}[htbp]
\begin{center}
\includegraphics[width = 8cm,angle=-90]{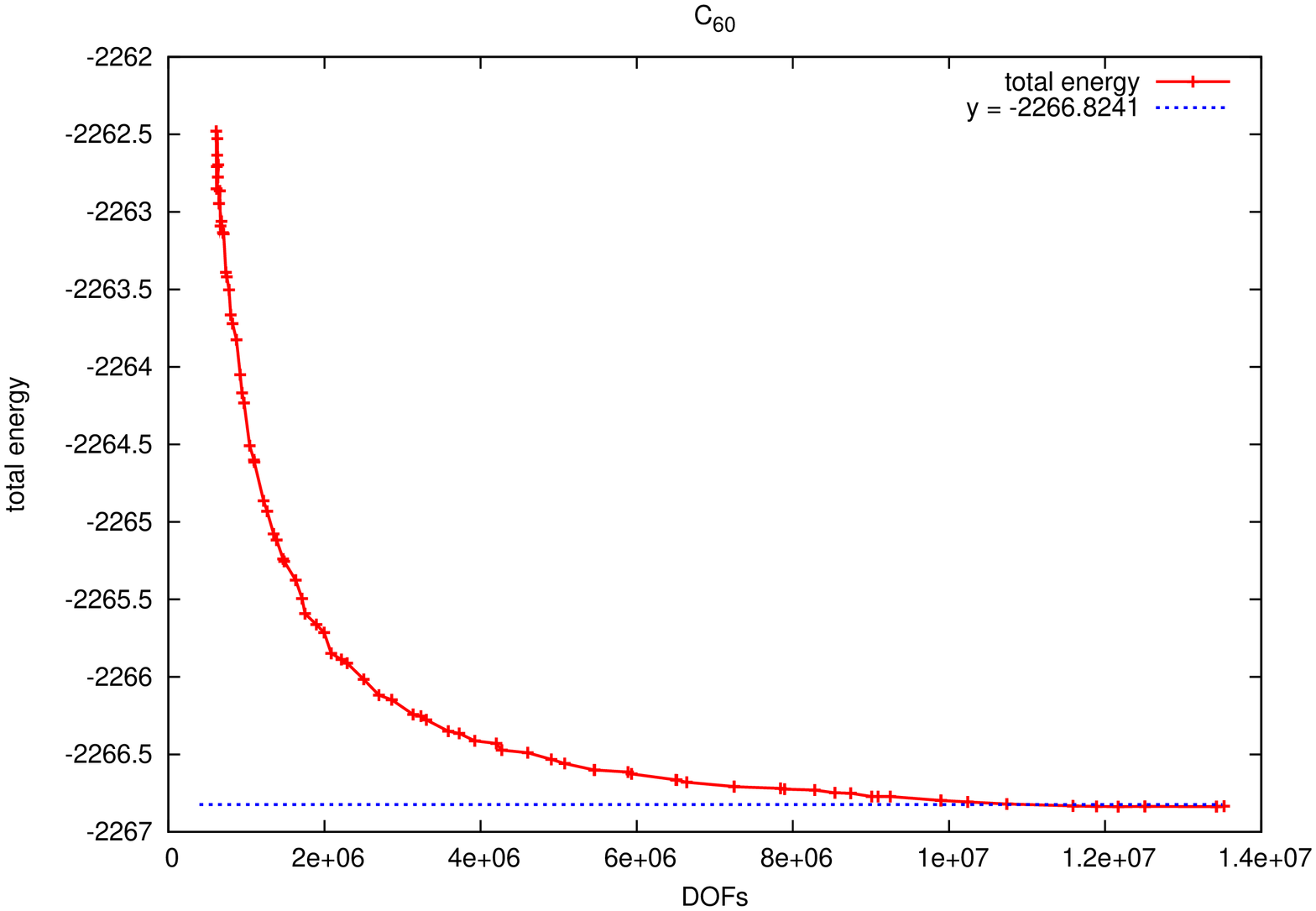} ~~
\end{center}
\caption{Convergence curves of the ground state energy}\label{fig-C60-energy}
\end{figure}\vskip 0.2cm

¡¡

\section{Concluding remarks}
In this paper, we proposed a parallel orbital-updating algorithm for
electronic structure calculations, which are demonstrated to be accurate and efficient for full-potential calculations
based on finite element discretizations.

From the comparison between results obtained by our algorithm and
those obtained by Gaussian 09 with different kinds of basis sets, we may
conclude that
\begin{itemize}
\item our algorithm can produce highly accurate approximations
to the exact one for medium or large scale systems;

\item  the cpu time cost by our algorithm is much
 lower than that cost by Gaussian09, especially for large scale systems;

\item our algorithm are efficient for full-potential calculations for  large scale systems;

\item our algorithm is
sequential parallel, which may mean that it has a potential to supercomputing.
\end{itemize}

Although  we have a primitive analysis as Theorem \ref{thm-estimate-h1}, we are currently not able to present a
mathematically rigorous
proof how  the approximations converge to the exact ground state energy and density.
Note that only  simple and basic  source problems have been employed in our
algorithms. It is our on-going work to study and apply more efficient source models into
our parallel orbital-updating approach
to speed up the approximation convergence, which will be addressed elsewhere.
Anyway, we believe that our approach is a general and powerful parallel-computing technique
that can be applied to a variety of eigenvalue problems, including partial differential equation based
ones with differential types of discretization methods, and other nonlinear problems.

\section*{Acknowledgements}  The authors would
like to thank Doctor Hongping Li his helps in using Gaussian09.

\section*{Appendix: Parallel orbital-updating algorithm for linear eigenvalue problems} \label{sec-algorithm-ep}
\setcounter{equation}{0}
\renewcommand{\theequation}{A.\arabic{equation}}
\renewcommand{\thetheorem}{A.\arabic{theorem}}
\renewcommand{\thelemma}{A.\arabic{lemma}}
\renewcommand{\theproposition}{A.\arabic{proposition}}
\renewcommand{\thesection}{A}
In this appendix, we show that the  parallel orbital-updating
approach can be applied to a general eigenvalue computation.

Consider eigenvalue problem:
\begin{equation}\Label{problem}
\left\{\begin{array}{rl}
L u &= \lambda u  \,\,\, \mbox{in} \quad \Omega, \\[0.2cm]
u &= 0\,\,\,\mbox{on}~~\partial\Omega, \end{array}\right.
\end{equation}
where  $L$ is a linear second order elliptic operator:
\[ Lu =-\nabla \cdot(A\nabla u)+cu\] with $A: \Omega\to \mathbb{R}^{d\times d}$ being piecewise Lipschitz
 with respect to the initial triangulation  and symmetric positive definite with smallest eigenvalue uniformly bounded away from 0, and $0\le c\in L^{\infty}(\Omega)$ \footnote{We mention that the results obtained in this paper are also valid for a more general bilinear form $a(\cdot,\cdot)$ (c.f., e.g., Remark 2.9 in \cite{dai-xu-zhou08}).}\label{footnote1}.

Define
\begin{eqnarray}
a(u, v) = \int_{\Omega}  A\nabla u \nabla v + c u v, ~\forall u, v
\in H^1(\Omega).
\end{eqnarray}


The weak form of (\ref{problem}) can be written as follows: find a pair  $(\lambda, u)\in
\mathbb{R}\times H^1_0(\Omega)$ satisfying $\|u\|_{0,\Omega}=1$ and
\begin{eqnarray}\label{eigen} a(u,v)=\lambda
(u,v)~~~~\forall v\in H^1_0(\Omega).
\end{eqnarray}

Equation (\ref{eigen}) has a  sequence of real eigenvalues $\sigma(L)\equiv\{\lambda_i: i=1,2,\cdots\}$:
$$
0 < \lambda_1\ <\lambda_2\le\lambda_3\le\cdots
$$
and the corresponding eigenfunctions
$
u_1, u_2, u_3,\cdots,
$
satisfying
$$
(u_i, u_j)=\delta_{ij}, ~i,j=1,2,\cdots,
$$
where the $\lambda_j$'s are repeated according to geometric
multiplicity.

Let $\lambda\in \sigma(L)$ and  $M(\lambda)$ denote the space of eigenfunctions corresponding to $\lambda$:
\begin{eqnarray*}
M(\lambda)=\{w\in H^1_0(\Omega): w ~\mbox{is an eigenvector of (\ref{eigen}) corresponding to } \lambda\}.
\end{eqnarray*}

A standard finite element  approximation for (\ref{eigen}) is: find a pair
$(\lambda_h, u_h)\in \mathbb{R}\times V_h$ satisfying
$\|u_h\|_{0,\Omega}=1$ and
\begin{eqnarray}\label{fem-eigen} a(u_h,v)=\lambda_h
(u_h,v)~~~~\forall v\in V_h.
\end{eqnarray}

We may order the eigenvalues of  Let $(\lambda_{i,h},u_{1,h})(i=1,2,\cdots,n_h\equiv \mbox{dim} V_h)$ be the
 eigenpairs of (\ref{fem-eigen})
satisfying
$$
b(u_{i,h}, u_{j,h})=\delta_{ij}, ~i,j=1,2,\cdots,n_h .
$$

We propose a parallel orbital-updating algorithm for solving
(\ref{fem-eigen}) as follows:

\begin{Algorithm}\label{basic-algorithm}
\begin{enumerate}
\item Given initial data $(\lambda^{(0)}_i, u_i^{(0)})\in \mathbb{R}\times H^1_0(\Omega)$
with $\|u^{(0)}_i\|_{0,\Omega}=1 (i=1,2,\cdots,N)$ and $\mathcal
{T}_0$, and let $n=0$.

\item  Construct $\mathcal {T}_{n+1}$ and $V_{n+1}$ based on an adaptive procedure  to $(\lambda^{(n)}_i, u_i^{(n)})$.

\item For $i=1,2,\cdots,N$, find $u^{(n+1/2)}_i\in V_{n+1}$ satisfying
\begin{eqnarray*} a(u^{(n+1/2)}_i, v)=
\lambda_i^{(n)}(u_i^{(n)},v) ~\forall v\in V_{n+1}
\end{eqnarray*}
in parallel.

\item Project to eigenspace: find $(\lambda^{(n+1)},u^{(n+1)})\in \mathbb{R}\times {\tilde V}_{n+1}$ satisfying
$\|u^{(n+1)}\|_{0,\Omega}=1$ and \begin{eqnarray*}
a(u^{(n+1)},v)=\lambda^{(n+1)} (u^{(n+1)},v) ~~\forall v\in {\tilde
V}_{n+1}
\end{eqnarray*}
to obtain eigenpairs
$(\lambda^{(n+1)}_i,u^{(n+1)}_i)(i=1,2,\cdots,N)$.

\item Let $n=n+1$ and go to Step 2.
\end{enumerate}
Here ${\tilde V}_{n+1}=~\mbox{span}~ \{u^{(n+1/2)}_1,u^{(n+1/2)}_2,\cdots,u^{(n+1/2)}_N\}$.
\end{Algorithm}\vskip 0.2cm

To analyze Algorithm \ref{basic-algorithm}, we introduce a
Galerkin-projection $P_h: H^1_0(\Omega) \to V_h\equiv S^h_0(\Omega)$
by
\begin{eqnarray}\label{Gprojection}
a(u-P_h u, v) =0 \quad\forall u\in H^1_0(\Omega) ~\forall  v\in V_h,
\end{eqnarray}
and apparently
\begin{equation*}\label{stable}
      \|P_hu\|_{1,\Omega}\lc \|u\|_{1,\Omega} \quad\forall u\in H^1_0(\Omega)
\end{equation*}
and
\begin{equation*}
      \|u-P_hu\|_{1,\Omega}\lc \inf_{v\in V_h}\|u-v\|_{1,\Omega} \quad\forall u\in
      H^1_0(\Omega).
\end{equation*}

\begin{theorem}\label{thm-estimate-h1}
Let $(\lambda_i^{(1)}, u_i^{(1)})$($i=1, \cdots, N$) be obtained by
Algorithm \ref{basic-algorithm} after one iteration,
 $(\lambda_i, u_i)$($i=1, \cdots, N$) be the first $N$ exact eigenpair of (\ref{eigen}). Then
\begin{eqnarray}\label{conc-estimate-h1}
d_{1}(u_i^{(1)}, V)  \lesssim \sum_{k=1}^{N}\Big(|\lambda_k- \lambda_k^{(0)}|+\inf_{v\in V_1}
\|u_k-v\|_1+\|u_k-u_k^{(0)}\|_0\Big),
\end{eqnarray}
where $V = span\{u_1, \cdots, u_N\}$, and the distance between $w$ and $W \in H_0^1(\Omega)$ is defined by
\begin{eqnarray}
 d_{1}(w,  W) = \inf_{v \in W} \| w - v\|_{1, \Omega}.
\end{eqnarray}
\end{theorem}
\begin{proof}
Let $P_1: H^1_0(\Omega)\longrightarrow V_1$ be the Galerkin
projection defined by (\ref{Gprojection}) when $V_h$ is replaced by
$V_1$. We see  for any $u\in H^1_0(\Omega)$ that
\begin{eqnarray*}
a(P_1 u_i - u_i^{(1/2)},v)
 =\lambda_i(u_i- u_i^{(0)},v)
+(\lambda_i - \lambda_i^{(0)})(u_i^{(0)}- ,v)\ \ \forall v\in V_1,
\end{eqnarray*}
which leads to
\begin{eqnarray*}
\|P_1 u_i - u_i^{(1/2)}\|_1&\lesssim\|u_i- u_i^{(0)}\|_0+|\lambda_i - \lambda_i^{(0)}|.
\end{eqnarray*}
We obtain from the triangle inequality that
\begin{eqnarray*}
\|u_i-u_i^{(1/2)}\|_1&\lesssim\|&u_i-P_1 u_i\|_1+\|P_1 u_i-u_i^{(1/2)}\|_1\\
    &\lesssim& \|u_i-P_1 u_i\|_1+\|u_i- u_i^{(0)}\|_0+|\lambda_i -
    \lambda_i^{(0)}|,
\end{eqnarray*}
where the first term on the right-hand side can be bounded by term
$\inf_{v\in V_1}\|u_i-v\|_1$.

Since  $u_i^{(1)} \in \tilde{V}_1 = span\{u_1^{(1/2)}, \cdots, u_N^{(1/2)}\}$, we have that there exist constants $\alpha_{i, k}(k=1, \cdots, N)$, such that
\begin{eqnarray*}
u_i^{(1)} = \sum_{k=1}^{N} \alpha_{i, k} u_k^{(1/2)}.
\end{eqnarray*}
Note that we may estimate as follows
\begin{eqnarray*}
\sum_{k=1}^{N} \alpha_{i, k} (u_k^{(1/2)} - u_k) &\lesssim &\sum_{k=1}^{N} \alpha_{i, k}  \Big(
\|u_k-P_1 u_k\|_1+\|u_k- u_k^{(0)}\|_0+|\lambda_k - \lambda_k^{(0)}|\Big) \\
 &\lesssim & \sum_{k=1}^{N} \Big(\|u_k-P_1 u_k\|_1+\|u_k- u_k^{(0)}\|_0+|\lambda_k - \lambda_k^{(0)}|\Big).
\end{eqnarray*}
Consequently,
\begin{eqnarray*}
 u_i^{(1)} - \sum_{k=1}^{N} \alpha_{i, k}  u_k  \lesssim  \sum_{k=1}^{N} \Big(\|u_k-P_1 u_k\|_1+\|u_k- u_k^{(0)}\|_0+|\lambda_k - \lambda_k^{(0)}|\Big),
\end{eqnarray*}
which means (\ref{conc-estimate-h1}) since $u=\sum_{k=1}^{N}
\alpha_{i, k}  u_k \in V$. This completes the proof.

\end{proof}

\end{document}